\documentclass[preprint,12pt]{larticle}
\usepackage[a4paper, total={6.18in, 10in}]{geometry}
\usepackage[utf8]{inputenc}

\usepackage{amssymb, amsmath, amsfonts, amssymb, amsthm}
\usepackage{mathtools}

\usepackage[symbol]{footmisc}

\newtheorem{theorem}{Theorem}
\newtheorem{lemma}{Lemma}
 
\theoremstyle{definition}

\newcommand{\set}[1]{\left\{#1\right\}}
\newcommand{\seq}[1]{\left[#1\right]}
\newcommand{\abs}[1]{\left\vert#1\right\vert}
\newcommand{\rep}[1]{\langle#1\rangle}

\newcommand{\ceil}[1]{{\left\lceil #1 \right\rceil}}

\newcommand{\Q}{\mathcal{Q}}
\newcommand{\T}{\mathcal{T}}
\newcommand{\K}{\mathcal{L}}
\newcommand{\C}{\mathcal{C}}
\newcommand{\X}{\mathcal{X}}

\begin{document}

\begin{frontmatter}

\title{A Generalized Hamming Distance of Sequence Patterns}

\author{Pengyu Liu\footnotesize{$^{1,2}$}\footnote{To whom correspondence should be addressed; e-mail: penliu@ucdavis.edu.}}

\author{Jingzhou Na\footnotesize{$^2$}}

\address{\footnotesize{$^1$}Department of Microbiology and Molecular Genetics, University of California, Davis, Davis, CA 95616}
\address{\footnotesize{$^2$}Department of Mathematics, Simon Fraser University, Burnaby, BC V5A 1S6, Canada}

\begin{abstract}

We define sequence patterns of length $n$ and level $\ell$ to be equivalence classes of sequences that have $n$ elements from the set of $\ell$ integer symbols $\set{1,2,\ldots,\ell}$ with no restriction on repetition, where the equivalence relation is induced by symbol relabeling without swapping positions of symbols. 
We define a distance for a set of $k$ sequence patterns of length $n$ and level $\ell$ by generalizing the Hamming distance between sequences.
We compute the maximal distance for $k$ sequence patterns of length $n$ and level $\ell$ and demonstrate how to calculate the exact distance between a pair of length-$n$ level-$\ell$ sequence patterns.

\end{abstract}

\end{frontmatter}


\section{Introduction}
\label{sec1}

Consider sequences of length $n$ whose elements are from a set of $\ell$ integer symbols $L = \set{1,2,\ldots,\ell}$ with no restriction on repetition. 
We say that the sequences are length-$n$ level-$\ell$ sequences.
We define length-$n$ level-$\ell$ sequence patterns to be equivalence classes of length-$n$ level-$\ell$ sequences under the action of the symmetric group of order $\ell$ on individual elements of the sequences.
Note that the group action only relabels symbols and does not change element positions. 
The Hamming distance between a pair of sequences of equal length is defined to be the number of positions with different elements in the sequences, which is an important tool in coding theory for error detecting \cite{Hamming1950}.
In this paper, we generalize the Hamming distance for a pair of length-$n$ level-$\ell$ sequence patterns to be the minimal Hamming distance between a pair of sequences from different sequence patterns. 
We compute the maximal Hamming distance for all pairs of length-$n$ level-$\ell$ sequence patterns and the exact Hamming distance between a pair of length-$n$ level-$\ell$ sequence patterns.
We also define the Hamming distance for a set of $k > 2$ sequence patterns of length $n$ and level $\ell$ and compute the maximal Hamming distance for all sets of $k$ sequence patterns of length $n$ and level $\ell$.

We note that a sequence in this paper is a finite word in the context of combinatorics on words, where the set of symbols $L$ is called an alphabet for the word \cite{Lothaire1997}. 
When $n = \ell$, a length-$n$ level-$n$ sequence pattern is also called a rhyme scheme for an $n$-line stanza, and it is known that the number of rhyme schemes for $n$-line stanzas is the Bell number~$B_n$ \cite{Riordan1979}.
Actually, when $n \leq \ell$, if we group positions with identical symbols in a length-$n$ level-$\ell$ sequence pattern as subsets, then we have a partition of set with $n$ elements \cite{Rota1964}. So, there is an one-to-one correspondence between partitions of a set with $n$ elements and length-$n$ level-$\ell$ sequence patterns when~$n \leq \ell$.
If $n > \ell$, then there must be symbol repetition in a length-$n$ level-$\ell$ sequence, so the one-to-one correspondence is restricted to the set partitions with at most $\ell$ subsets.
Although sequence patterns are equivalent to set partitions, we find that the terminology of sequences is more convenient when discussing the Hamming distance. 
Therefore, we use the term ``sequence pattern'' throughout the paper.

\section{Sequence patterns}
\label{sec2}

\subsection{Definitions}
Let $n$ and $\ell$ be two positive integers and $N = \set{1,2,3,\ldots,n}$ and $L = \set{1,2,3,\ldots,\ell}$ be the sets of positive integers no greater than $n$ and $\ell$ respectively.
We call $N$ the {\em index set} and an element in $N$ an {\em index}.
We call $L$ the {\em symbol set} and an element in $L$ a {\em symbol}.
We define a {\em sequence} to be a function $q:N\to L$, where $n$ is the {\em length} of the sequence, and $\ell$ is the {\em level} of the sequence. 
Note that the level of a sequence $q$ is the number of possible symbols present in $q(N)$ rather than the actual number of symbols appeared in $q(N)$, and the level of a sequence is determined by the symbol set $L$ instead of the image $q(N)$. 
We denote a sequence by $q = \seq{x_1;x_2;\ldots;x_n}$, and we say that each $x_i$ is an {\em element} of the sequence.
We separate the elements of a sequence by semicolons to indicate that we write a sequence vertically as a column.
In this paper, letters $p,q$ will be used to denote sequences and $x,y$ will be used to denote the elements in a sequence.

Let $S_\ell$ be the symmetric group of order~$\ell$. 
Two length-$n$ level-$\ell$ sequences $q = \seq{x_1;x_2;\ldots;x_n}$ and $p = \seq{y_1;y_2;\ldots;y_n}$ are {\em equivalent} if there exists a permutation $\phi\in S_\ell$ such that $\phi(q) = \seq{\phi(x_1);\phi(x_2);\ldots;\phi(x_n)} = \seq{y_1;y_2;\ldots;y_n} = p$.
For instance, among the three length-$5$ level-$3$ sequences in Example (\ref{exeq}), the sequences $q_1$ and $q_2$ are equivalent under the permutation $\phi = (123) \in S_3$, that is $\phi(q_1) = q_2$, while no permutation can map $q_1$ or $q_2$ to $q_3$, so $q_3$ is not equivalent to $q_1$ or~$q_2$.
Here, the permutation $\phi = (123)$ is written in cycle notation, and we use the cycle notation for permutations throughout the paper unless otherwise stated.

\begin{equation}
q_1 = \begin{bmatrix} 1 \\ 1 \\ 3 \\ 2 \\ 1 \end{bmatrix} \quad q_2 = \begin{bmatrix} 2 \\ 2 \\ 1 \\ 3 \\ 2 \end{bmatrix} \quad q_3 = \begin{bmatrix} 1 \\ 1 \\ 2 \\ 2 \\ 1 \end{bmatrix}
\label{exeq}
\end{equation}

Let $\Q_{n,\ell}$ be the set of all length-$n$ level-$\ell$ sequences.
The equivalence relation induced by permutations in $S_\ell$ partitions~$\Q_{n,\ell}$.
We define a {\em sequence pattern} to be an equivalence class of~$\Q_{n,\ell}$.
We denote a sequence pattern by $t$ or $\rep{q}$, where $q$ is a sequence in the sequence pattern~$t$.
We say that the sequence $q$ is a {\em representative} of the sequence pattern~$t$, and that the sequence pattern $t$ is {\em generated} by $q$.
We define the {\em length} of a sequence pattern $\rep{q}$ to be the length of $q$, and similarly, the {\em level} of the sequence pattern $\rep{q}$ to be the level of $q$.
We denote the set of all length-$n$ level-$\ell$ sequence patterns by~$\T_{n,\ell}$.
For instance, the six sequences in Example (\ref{exseqtop}) form a length-$3$ level-$3$ sequence pattern.

\begin{equation}
q_1 = \begin{bmatrix} 1 \\ 1 \\ 2  \end{bmatrix} \quad q_2 = \begin{bmatrix} 1 \\ 1 \\ 3 \end{bmatrix} \quad q_3 = \begin{bmatrix} 2 \\ 2 \\ 1 \end{bmatrix} \quad q_4 = \begin{bmatrix} 2 \\ 2 \\ 3  \end{bmatrix} \quad q_5 = \begin{bmatrix} 3 \\ 3 \\ 1 \end{bmatrix} \quad q_6 = \begin{bmatrix} 3 \\ 3 \\ 2 \end{bmatrix}
\label{exseqtop}
\end{equation}

We say that a sequence is {\em constant} if all of its elements are identical, and a sequence pattern is {\em constant} if it contains a constant sequence.

\subsection{Enumeration}

We count the number of length-$n$ level-$\ell$ sequence patterns and denote the number by~$\abs{\T_{n,\ell}}$.
We say that a length-$n$ level-$\ell$ sequence $q: N\to L$ is in {\em standard order} if for any $i \in N$ with $q(i) > 1$, there exists a $j < i \in N$ such that $q(j) = q(i) - 1$.
We say that $q$ is a {\em standard sequence} of length $n$ and level $\ell$.
For instance, only $q_3$ among sequences in Example~(\ref{exeq}) is a standard sequence, and only $q_1$ among sequences in Example~(\ref{exseqtop}) is a standard sequence.
It is easy to show that every length-$n$ level-$\ell$ sequence can be mapped to a standard sequence by a permutation in $S_{\ell}$, and that there exists no permutation in $S_{\ell}$ that can map one standard sequence to a different standard sequence. 
Hence, every length-$n$ level-$\ell$ sequence pattern contains one and only one standard sequence, and counting the number of length-$n$ level-$\ell$ sequence patterns is equivalent to counting the number of standard sequences of length $n$ and level $\ell$. 
We denote the set of all length-$n$ level-$\ell$ standard sequences by $Q^*_{n,\ell}$ and the number of length-$n$ level-$\ell$ standard sequences by~$|Q^*_{n,\ell}|$.
We have $\abs{\T_{n,\ell}} = |Q^*_{n,\ell}|$,  and the number $|Q^*_{n,\ell}|$ is computed using Stirling numbers of the second kind~\cite{Arndt2016}.

\begin{theorem}[Arndt-Sloane \cite{Arndt2016}]\label{theorenum2}
The number of length-$n$ level-$\ell$ sequences that are in standard order is given by Formula~(\ref{enum2}). 

\begin{equation}
|Q^*_{n,\ell}| = \sum_{m=1}^{\ell}  \sum_{i=0}^{m} \frac{(-1)^i}{i!} \frac{(m-i)^n}{(m-i)!}
\label{enum2}
\end{equation}
\end{theorem}

Here, we present an alternative enumeration from the perspective of sequence patterns.
We note that a sequence pattern $t\in\T_{n,\ell}$ can be considered as the orbit of a sequence $q\in t$ under the action of the symmetric group $S_\ell$ on individual elements of~$q$.
Thus, we can compute $\abs{\T_{n,\ell}}$ with the orbit-counting theorem, also known as Cauchy-Frobenius lemma or Burnside's lemma~\cite{Gallian2010}. 

\begin{theorem}\label{theorenum}
The number of length-$n$ level-$\ell$ sequence patterns is given by Formula (\ref{enum}).
In particular, we have $\abs{\T_{n,2}} = 2^{n-1}$.

\begin{equation}
\abs{\T_{n,\ell}} = \frac{\ell^n}{\ell!} + \sum_{m=2}^{\ell} \frac{(\ell-m)^n}{(\ell-m)!} \sum_{i=0}^{m} \frac{(-1)^i}{i!}
\label{enum}
\end{equation}
\end{theorem}

\begin{proof}
The orbit-counting theorem states that the number of orbits can be computed with Formula (\ref{oct}), where $\Q_{n,\ell}^{\phi}=\set{q\in \Q_{n,\ell} \, | \, \phi(q) = q}$ is the set of fixed sequences by the permutation $\phi\in S_\ell$~\cite{Gallian2010}. 

\begin{equation}
\abs{\T_{n,\ell}} = \frac{1}{\abs{S_\ell}}\sum_{\phi \in S_\ell}\abs{\Q_{n,\ell}^{\phi}}
\label{oct}
\end{equation}

\noindent
The identity in $S_\ell$ fixes all sequences in $\Q_{n,\ell}$, so the number of sequences fixed by the identity is~$\ell^n$.
Let $S_{\ell|m} \subset S_\ell$ be the subset of permutations that derange $m$ symbols, where $2\leq m\leq \ell$.
The number of permutations in $S_{\ell|m}$ can be computed by counting the number of ways of selecting $m$ symbols from the symbol set $L$ and multiplying by the number of derangements of the $m$ symbols. See Formula (\ref{derangem}).

\begin{equation}
\abs{S_{\ell|m}} = \binom{\ell}{m} \left(m!\sum_{i=0}^{m} \frac{(-1)^i}{i!} \right) = \frac{\ell!}{(\ell-m)!}\sum_{i=0}^{m}\frac{(-1)^i}{i!}
\label{derangem}
\end{equation}

\noindent
We note that a permutation $\phi \in S_{\ell|m}$ fixes all and only sequences without the $m$ symbols that $\phi$ deranges, and there are $(\ell-m)^n$ such sequences in $\Q_{n,\ell}^{\phi}$. 
Therefore, we have Formula (\ref{eq}).

\begin{equation}
\sum_{\phi \in S_\ell}\abs{\Q_{n,\ell}^{\phi}} = \ell^n + \ell! \sum_{m = 2}^{\ell} \frac{(\ell-m)^n}{(\ell-m)!}\sum_{i=0}^{m}\frac{(-1)^i}{i!}
\label{eq}
\end{equation}

\noindent
Then Formula (\ref{enum}) follows applying Formula (\ref{eq}) to Formula (\ref{oct}).
\end{proof}

Recall that when~$n \leq \ell$, there is an one-to-one correspondence between sequence patterns and partitions of set with $n$ elements.
So, we have $\abs{\T_{n,\ell}} = B_n$ for any $n \leq \ell$, where $B_n$ is the $n$-th Bell number.
If we set $n = \ell$, then Formula (\ref{enum}) is equivalent to Dobi\'nski's formula for Bell numbers \cite{Pitman1997}. 
When $n > \ell$, the number $\abs{\T_{n,\ell}}$ equals the number of partitions of an $n$-element set with at most $\ell$ subsets.

\section{Hamming distance of sequence patterns}\label{sec3}

\subsection{Basic definitions}
Let $Q = \set{q_1, q_2, \ldots, q_k} \subset \Q_{n,\ell}$ be a set of sequences with $k \geq 2$. 
Recall that a sequence is a function from the index set $N = \set{1,2,3,\ldots,n}$ to the symbol set $L = \set{1,2,3,\ldots,\ell}$.
We define the {\em Hamming distance} (or simply the {\em distance}) of sequences in $Q$ to be the number of indices in $N$ whose images in $q_1$, $q_2$, \ldots, $q_k$ are not identical. 
We denote the distance of the sequences in $Q$ by $d(Q)$ or $d(q_1, q_2, \ldots, q_k)$. 
For instance, the distance of the three length-$5$ level-$3$ sequences in  Example (\ref{exdiff}) is $d(q_1,q_2,q_3) = 4$, because only the fourth elements in the sequences are identical.

\begin{equation}
q_1 = \begin{bmatrix} 1 \\ 1 \\ 3 \\ 2 \\ 1 \end{bmatrix} \quad q_2 = \begin{bmatrix} 3 \\ 3 \\ 1 \\ 2 \\ 3 \end{bmatrix} \quad q_3 = \begin{bmatrix} 1 \\ 1 \\ 2 \\ 2 \\ 1 \end{bmatrix}
\label{exdiff}
\end{equation}

\noindent
Note that the order of sequences in $Q$ is irrelevant in computing the distance.
So to simplify our arguments, we fix the order of sequences and write $Q = \seq{q_1,q_2,\ldots,q_k}$ as a sequence (of sequences). 
Here, we separate the sequences in $Q$ by colons to indicate that we list the sequences horizontally.
For example, in computing the distance of the three sequences $q_1,q_2$ and $q_3$ in Example (\ref{exdiff}), we write the set $Q$ of the three sequences as in Example (\ref{excross}).

\begin{equation}
Q = [q_1,q_2,q_3] = \begin{bmatrix} 1 & 3 & 1 \\ 1 & 3 & 1 \\ 3 & 1 & 2 \\ 2 & 2 & 2 \\ 1 & 3 & 1 \end{bmatrix}
\label{excross}
\end{equation}

Let $K = \set{1,2,\ldots,k}$ be the set of $k$ positive integers.
We assume $k \geq 2$ throughout the paper unless otherwise stated.
We define a {\em cross section} of $Q = \seq{q_1,q_2,\ldots,q_k}$ to be a sequence $c_i : K \to L$ consisting of the $i$-th elements of the sequences in $Q$, that is $c_i = \seq{q_1(i),q_2(i),\ldots,q_k(i)}$.
We denote the set of length-$k$ level-$l$ cross sections by $\C_{k,\ell}$, to distinguish the horizontally written sequences (cross sections) from the vertically written sequences in $\Q_{k,\ell}$. 
Similarly, we say a cross section is {\em constant} if all of its elements are identical.
Let $s(Q)$ be the number of constant cross sections in $Q$.
We have $d(Q) = n - s(Q)$.
For instance, the sequences $c_1 = [1,3,1]$, $c_2 = [1,3,1]$, $c_3 = [3,1,2]$, $c_4 = [2,2,2]$ and $c_5 = [1,3,1]$ are cross sections of $Q$ displayed in Example (\ref{excross}), and $c_4$ is the only constant cross section in $Q$, so $s(Q) = 1$ and $d(Q)= 4$.

Let $T = \set{t_1, t_2, \ldots, t_k} \subset \T_{n,\ell}$ be a set of sequence patterns.
Recall that a sequence pattern $t$ is an equivalence class of sequences under permutations, and $t$ can be denoted by $\rep{q}$ with a representative $q\in t$.
We define the {\em Hamming distance} (or simply the {\em distance}) of sequence patterns in $T$ by Formula (\ref{diff}) and denote the distance by $d(T)$ or $d(t_1,t_2, \ldots,t_k)$.

\begin{equation}
d(t_1,t_2,\ldots,t_k) = \min_{q_i\in t_i}d(q_1,q_2,\ldots,q_k)
\label{diff}
\end{equation}

\noindent
Suppose $t_1$, $t_2$, \ldots, $t_k$ are respectively generated by sequences $q_1$, $q_2$, \ldots, $q_k$ in $\Q_{n,\ell}$, and we write the set $Q = \seq{q_1,q_2,\ldots,q_k}$ as a sequence. 
We can analogously write the set of sequence patterns as $T = \rep{Q} = \seq{\rep{q_1},\rep{q_2},\ldots,\rep{q_k}} = \seq{t_1,t_2,\ldots,t_k}$. 
The distance of $T = \seq{t_1,t_2,\ldots,t_k}$ can also be defined using permutations. 
Let $\Phi = \seq{\phi_1,\phi_2,\ldots,\phi_k} \in S_\ell^k$ be a sequence of permutations, where $S_\ell^k$ is the Cartesian product of $k$ symmetric groups of order $\ell$.
We define $\Phi(Q) = \seq{\phi_1(q_1),\phi_2(q_2),\ldots,\phi_k(q_k)}$ and the distance for a set $T = \rep{Q}$ of sequence patterns by Formula~(\ref{diffq}).

\begin{equation}
d(\rep{Q}) = \min_{\Phi\in S_\ell^k} d(\Phi(Q))
\label{diffq}
\end{equation}

\noindent
Note that the distance of $Q$ can be computed by counting the number of constant cross sections in $Q$. So the distance of $T = \rep{Q}$ can also be computed by Formula~(\ref{diffx}).

\begin{equation}
d(\rep{Q}) = n - \max_{\Phi\in S_\ell^k} s(\Phi(Q))
\label{diffx}
\end{equation}

\noindent
For instance, the sequence of permutations $\Phi = \seq{(1),(13),(1)}$ maps $Q$ in Example (\ref{excross}) to $\Phi(Q)$ in Example (\ref{excons}), and there are 4 constant cross sections in $\Phi(Q)$. 
It is easy to check that the maximal number of constant cross sections in $\Phi(Q)$ is 4 for any $\Phi\in S_{\ell}^k$. 
So the distance for the set $\rep{Q}$ of sequence patterns is~$1$. 

\begin{equation}
\Phi(Q) = \begin{bmatrix} 1 & 1 & 1 \\ 1 & 1 & 1 \\ 3 & 3 & 2 \\ 2 & 2 & 2 \\ 1 & 1 & 1 \end{bmatrix}
\label{excons}
\end{equation}

\subsection{Metric spaces of sequence patterns}

It is well known that the Hamming distance between two sequences in $\Q_{n,\ell}$ is a metric. 
Namely, the Hamming distance has the following three properties.
\begin{enumerate}
\item Identity: For any sequences $q,p\in \Q_{n,\ell}$, we have $d(q,p) = 0$ if and only if~$q = p$. 
\item Symmetry: For any sequences $q,p\in \Q_{n,\ell}$, we have $d(q,p) = d(p,q)$;
\item Triangle inequality: For any sequences $q_1,q_2,q_3\in \Q_{n,\ell}$, the inequality $d(q_1,q_3) + d(q_2,q_3) \geq d(q_1,q_2)$ holds.
\end{enumerate}
It is trivial to check that the generalized Hamming distance of two sequence patterns in $\T_{n,\ell}$ satisfies the first two properties.
Here, we show that the triangle inequality holds.
Consider three sequence patterns $t_1 = \rep{q_1}, t_2 = \rep{q_2}$ and $t_3 = \rep{q_3}$ in~$\T_{n,\ell}$.
Suppose that $d(t_1,t_3) = d(\phi_1(q_1),q_3)$ for a permutation $\phi_1\in S_\ell$ and that $d(t_2,t_3) = d(\phi_2(q_2),q_3)$ for a permutation~$\phi_2\in S_\ell$.
We have $d(t_1,t_3) + d(t_2,t_3) = d(\phi_1(q_1),q_3) + d(\phi_2(q_2),q_3) \geq d(\phi_1(q_1),\phi_2(q_2))$. 
Note that $d(t_1,t_2)$ is defined to be the minimal distance of $d(\phi_a(q_1),\phi_b(q_2))$ over all $\phi_a,\phi_b\in S_\ell$. 
We have $d(\phi_1(q_1),\phi_2(q_2)) \geq d(t_1,t_2)$ and $d(t_1,t_3) + d(t_2,t_3) \geq d(t_1,t_2)$.
Therefore, we have the following theorem.

\begin{theorem}\label{metric}
The generalized Hamming distance defined for a pair of sequences patterns in $\T_{n,\ell}$ is a metric, and all sequence patterns in $\T_{n,\ell}$ together with the generalized Hamming distance for two sequence patterns form a metric space.
\end{theorem}

\subsection{Bounds of Hamming distance}
It is clear that the minimal distance for a set of sequence patterns is~$0$. 
The {\em maximal distance} of $k$ sequence patterns in $\T_{n,\ell}$, denoted by $D_{n,\ell,k}$, is defined by Formula~(\ref{maxdiff}).

\begin{equation}
D_{n,\ell,k} = \max_{t_i\in \T_{n,\ell}}d(t_1,t_2,\ldots,t_k)
\label{maxdiff}
\end{equation}

For any set $Q = \seq{q_1, q_2, \ldots, q_k}$ of sequences in $\Q_{n,\ell}$, we can always apply a sequence of permutations $\Phi\in S_\ell^k$ to $Q$ such that every sequence in $\Phi(Q)$ has the same first element. 
Hence, the maximal distance $D_{n,\ell,k}$ has a trivial upper bound.

\begin{lemma}\label{upb}
For any positive integers $n$, $\ell$ and $k\geq 2$, we have $D_{n,\ell,k} \leq n-1$.
\end{lemma}

We characterize the sets of $k$ sequence patterns in $\T_{n,\ell}$ that have distance~$n-1$.
Let $c = \seq{a_1,a_2,\ldots,a_k}$ and $c' = \seq{b_1,b_2,\ldots,b_k}$ be two cross sections in $\C_{k,\ell}$.
We say that $c$ and $c'$ are {\em identical} if $a_i=b_i$ for all $1 \leq i \leq k$.
We say that $c$ and $c'$ are {\em incompatible} if $a_i \neq b_i$ for all $1 \leq i \leq k$. 
We say that $c$ and $c'$ are {\em connected} (by a sequence of permutations) if they are either identical or incompatible. 

\begin{lemma}\label{mcs}
Let $T = \rep{Q} = \seq{\rep{q_1},\rep{q_2},\ldots,\rep{q_k}}$ be a set of $k$ sequence patterns in $\T_{n,\ell}$.
Then $d(T)= n-1$ if and only if $Q$ contains no pair of connected cross sections.
\end{lemma}

\begin{proof}
If there are two identical cross sections $c = c' = \seq{a_1,a_2,\ldots,a_k}$ in $Q$, then the sequence of permutations $\Phi^* = \seq{(a_1),(a_1a_2),\ldots,(a_1a_k)}$ creates two constant cross sections in $\Phi^*(Q)$. 
So, we have $d(T) < n-1$, contradicting the assumption. 
If there are two incompatible cross sections $c = \seq{a_1,a_2,\ldots,a_k}$ and $c' = \seq{b_1,b_2,\ldots,b_k}$ in~$Q$, then the sequence of permutations $\Phi^* = \seq{\phi^*_1,\phi^*_2,\ldots,\phi^*_k}$ creates two constant cross sections in $\Phi^*(Q)$, where $\phi^*_i$ is given by the two-line notation in Formula (\ref{sperm}).

\begin{equation}
\phi^*_i = \begin{pmatrix}
 a_i & b_i & \cdots  \\
 a_1 & b_1 & \cdots
\end{pmatrix}
\label{sperm}
\end{equation}

\noindent
Note that $c$ and $c'$ are incompatible, so we have $a_1 \neq b_1$ and $a_i \neq b_i$, which guarantee that the sequence of permutations $\Phi^*$ is well defined. 
Similarly, two constant cross sections in $\Phi(Q)$ imply $d(T) < n-1$, which contradicts the assumption. 

Conversely, assume $d(T) < n-1$, then there exists a pair of cross sections $c = \seq{a_1,a_2,\ldots,a_k}$ and $c' = \seq{b_1,b_2,\ldots,b_k}$ in $Q$ such that $\Phi(c)$ and $\Phi(c')$ are constant in $\Phi(Q)$ for a sequence of permutations $\Phi = \seq{\phi_1,\phi_2,\ldots,\phi_k} \in S_\ell^k$.
We claim that $c$ and $c'$ are either identical or incompatible. 
If they are not identical or incompatible, then there exists an index $i\in N$ such that $a_i = b_i$, and there also exists a different index $j \in N$ such that $a_j \neq b_j$. 
Since the cross sections $\Phi(c)$ and $\Phi(c')$ are constant in~$\Phi(Q)$, we have $\phi_i(a_i) = \phi_j(a_j)$ and $\phi_i(b_i) = \phi_j(b_j)$.
Furthermore, because $a_i = b_i$, we have $\phi_i(a_i) = \phi_i(b_i)$.
These imply that $\phi_j(a_j) = \phi_j(b_j)$, which contradicts $a_j \neq b_j$.
\end{proof}

The proof of Lemma~\ref{mcs} can be generalized for a set of pairwise connected cross sections.

\begin{lemma}\label{cn}
Let $Q = \seq{q_1,q_2,\ldots,q_k}$ be a set of $k$ sequences in $\Q_{n,\ell}$.
If $Q$ contains $m$ pairwise connected cross sections, then there exists a sequence of permutations $\Phi^*\in S_\ell^k$ such that the $m$ pairwise connected cross sections are mapped to $m$ constant cross sections by $\Phi^*$.
\end{lemma}

\begin{proof}
For any $1\leq i\leq m$, let $c^i = \seq{a^i_1,a^i_2,\ldots,a^i_k}$ be one of the $m$ pairwise connected cross sections in~$Q$. 
Since connected cross sections can be identical, we assume that there are $w\leq m$ unique cross sections among them. 
Without loss of generality, let $c^1,c^2,\ldots,c^w$ be the unique cross sections.
Because they are pairwise incompatible, the elements $a^1_j$,$a^2_j$,\ldots$a^w_j$ are $w$ distinct symbols for any $1\leq j\leq k$. We define a sequence of permutations $\Phi^* = \seq{\phi^*_1,\phi^*_2,\ldots,\phi^*_k} \in S_\ell^k$ by the two-line notation in Formula~(\ref{tlperm})

\begin{equation}
\phi^*_j = \begin{pmatrix}
a^1_j & a^2_j & \dots& a^{w-1}_j & a^w_j & \cdots\\
a^1_1 & a^2_1 & \dots& a^{w-1}_1 & a^w_1 & \cdots
\end{pmatrix}
\label{tlperm}
\end{equation}

\noindent
Note that $c^1,c^2,\ldots,c^w$ being pairwise incompatible implies that the sequence of permutations $\Phi^*$ is well defined.
It is trivial to check that the $m$ connected cross sections are mapped to $m$ constant cross sections by $\Phi^*$.
\end{proof}

We examine the connectedness of cross sections in $\C_{k,\ell}$.
There are in total $\ell^k$ cross sections in~$\C_{k,\ell}$. 
We divide $\C_{k,\ell}$ into $\ell$ subsets based on their first elements, and we denote the subset of cross sections with first element $i$ by $\C_{k,\ell|i}$.
Let $c$ be a cross section in $\C_{k,\ell|i}$ and $c'$ be a cross section in~$\C_{k,\ell|i+1}$.
We say that $c$ is {\em linked} to $c'$ if $c' = \psi(c)$, where $\psi = (12\ldots\ell) \in S_\ell$. 
Let $c$ be a cross section in~$\C_{k,\ell|1}$.
We define a {\em link} $\K(c)$ generated by $c$ to be a subset of $\C_{k,\ell}$ such that $\K(c) = \set{c,\psi(c),\psi^2(c),\ldots,\psi^{\ell-1}(c)}$.
For instance, we display a link of $\C_{5,3}$ in Example (\ref{exlink}).

\begin{equation}
c = [1,2,3,1,1] \quad \psi(c) = [2,3,1,2,2] \quad \psi^2(c) = [3,1,2,3,3]
\label{exlink}
\end{equation}

\noindent
It is clear that for any cross section $c\in \C_{k,\ell|1}$, every subset $\C_{k,\ell|i}$ has one and only one element in~$\K(c)$. 
It is also trivial that for different cross sections $c,c'\in \C_{k,\ell|1}$, we have $\K(c)\cap \K(c') = \emptyset$.
So the links partition $\C_{k,\ell}$, and there are $\ell^{k-1}$ links in~$\C_{k,\ell}$.
Moreover, since $\psi$ increases every symbol by $1$ in $\mathbb{Z}_\ell$ (where we set $0 = \ell$), every pair of cross sections in a link is incompatible, hence connected. 

We say that a set $Q$ of $k$ sequences in $\Q_{n,\ell}$ is {\em complete} if $\Phi(Q)$ contains one and only one constant cross section for any sequence of permutations~$\Phi\in S_\ell^k$.
We say that $Q$ is {\em semi-complete} if $\Phi(Q)$ contains at most one constant cross section for any sequence of permutations~$\Phi\in S_\ell^k$.
Note that there exist totally $\ell^{k-1}$ cross sections in $\C_{k,\ell|1}$. 
We list these cross sections in lexicographic order of the elements in them and construct a set $M = \seq{q^*_1,q^*_2,\ldots,q^*_k}$ of sequences in $\Q_{\ell^{k-1},\ell}$ such that the cross sections in $M$ are the cross sections from $\C_{k,\ell|1}$ in lexicographic order. See Formula (\ref{mm}).
For a positive integer $r<\ell^{k-1}$, we define $M_r$ to be the set of $k$ sequences in $\Q_{r,\ell}$ that contains only the first $r$ cross sections in $M$.

\begin{equation}
M = \seq{q^*_1,q^*_2,\ldots,q^*_k} = \begin{bmatrix}
1 & 1 & \ldots & 1 & 1\\
1 & 1 & \ldots & 1 & 2\\
 \vdots &  \vdots & \ddots &  \vdots & \vdots\\
1 & 1 & \ldots & 1 & \ell\\
1 & 1 & \ldots & 2 & 1\\
 \vdots &  \vdots & \ddots &  \vdots & \vdots\\
 1 & \ell & \ldots & \ell & \ell\\
\end{bmatrix}
\label{mm}
\end{equation}

\begin{lemma}\label{comp}
The set $M$ of $k$ sequences in $\Q_{\ell^{k-1},\ell}$ is complete, and the set $M_r$ of $k$ sequences in $\Q_{r,\ell}$ is semi-complete.
\end{lemma}
\begin{proof}
Because $q^*_1$ is a constant sequence and there is no identical cross sections in $M$ or $M_r$, no pair of cross sections in $M$ or $M_r$ is connected.
By Lemma~\ref{upb} and Lemma~\ref{mcs}, we have $d(\rep{M}) = \ell^{k-1} - 1$, and $d(\rep{M_r}) = r - 1$. 
So $M$ and $M_r$ are semi-complete.

Let $M' = \seq{q^*_2,q^*_3,\ldots,q^*_k}$ be a set of $k-1$ sequences in~$\Q_{\ell^{k-1},\ell}$. 
The cross sections in $M'$ contain all combinations of assigning $\ell$ symbols to $k-1$ positions.
So there are $\ell$ unique constant cross sections in $M'$, namely $\seq{1,1,\ldots,1}$, $\seq{2,2,\ldots,2}$,\ldots,$\seq{\ell,\ell,\ldots,\ell}$.
For any sequence of permutations $\Phi' = \seq{\phi_2,\phi_3,\ldots,\phi_k}\in S_\ell^{k-1}$, these constant cross sections persist in~$\Phi'(M')$. 
So for any sequence of permutations~$\Phi = \seq{\phi_1,\phi_2,\ldots,\phi_k}$ in~$S_\ell^k$, no matter what $\phi_1(1)\in L$ in $\Phi(M)$ is, there always exists a constant cross section in $\Phi(M)$.
Therefore, $M$ is complete.
\end{proof}

\begin{theorem}\label{nlk}
The maximal distance of $k$ sequence patterns of length $n$ and level $\ell$ is given by Formula~(\ref{mdk}).

\begin{equation}
D_{n,\ell,k} = n - \ceil{\frac{n}{\ell^{k-1}}}
\label{mdk}
\end{equation}
\end{theorem}

\begin{proof}
For any set $Q$ of $k$ sequences in $\Q_{n,\ell}$, there exists a link $\K(c)$ generated by a cross section $c\in \C_{k,\ell|1}$ such that $Q$ has at least $\ceil{n/\ell^{k-1}}$ cross sections of $\K(c)$ due to the pigeonhole principle. 
These cross sections in $\K(c)$ are connected, so they are constant cross sections in $\Phi(Q)$ for a sequence of permutations $\Phi\in S_\ell$ by Lemma~\ref{cn}.
Therefore, we have $s(\Phi(Q)) \geq \ceil{n/\ell^{k-1}}$ and $D_{n,\ell,k} \leq n - \ceil{n/\ell^{k-1}}$.

Suppose that $n = m\ell^{k-1}+r$ for integers $m \geq 0$ and $0\leq r<\ell^{k-1}$. 
Consider the set $M_n = [M;M;\ldots;M;M_r]$ of $k$ sequences in $\Q_{n,\ell}$ constructed by vertically concatenating $m$ copies of $M$ and one copy of $M_r$. 
By Lemma~\ref{comp}, the set $M$ is complete and the set $M_r$ is semi-complete.
So, there are at most $m+1$ constant cross sections in $\Phi(M_n)$ for any sequence of permutations~$\Phi\in S_e^k$.
Therefore, we have $d(\rep{M_n}) = n - (m + 1) = n - \ceil{n/\ell^{k-1}}$ and $D_{n,\ell,k} = n - \ceil{n/\ell^{k-1}}$.
\end{proof}

\subsection{Computing exact Hamming distance}

Let $T = \rep{Q} = \seq{\rep{q_1},\rep{q_2},\ldots,\rep{q_k}}$ be a set of sequence patterns in $\T_{n,\ell}$ and $C(Q) = \seq{c_1;c_2;\ldots;c_n}$ be the sequence of all cross sections in~$Q$. 
We say that a subset $X(Q) \subset C(Q)$ is {\em maximally connected} if all the cross sections in $X(Q)$ are pairwise connected, and any cross section in $C(Q)-X(Q)$ is connected to some cross section in~$X(Q)$.
We use $\X_i(Q)$ to denote the set of all maximally connected subsets of $C(Q)$ that contain $c_i$, and we use $v(\X_i(Q))$ to denote the maximum cardinality of elements in~$\X_i(Q)$.
Lemma~\ref{mcs} implies that two cross sections that are not connected can not both be mapped to constant cross sections by any sequence of permutations~$\Phi\in S_\ell^k$.
So, the minimum number of cross sections in $C(Q)-X(Q)$ over all maximally connected subset $X(Q) \subset C(Q)$ gives the exact distance $d(\rep{Q})$, and we have the following Lemma~\ref{xiq}, where Formula~(\ref{eqxiq}) can be used to compute the exact distance.
\begin{lemma}\label{xiq}
Let $Q = \seq{q_1,q_2,\ldots,q_k}$ be a set of $k$ sequences in $\Q_{n,\ell}$. 

\begin{equation}
d(\rep{Q}) = n - \max_{\Phi\in S_\ell^k} s(\Phi(Q)) = n - \max_{1\leq i \leq n} v(\X_i(Q))
\label{eqxiq}
\end{equation}
\end{lemma}

Let $T = \rep{Q} = \seq{\rep{q_1},\rep{q_2}}$ be a set of two sequence patterns in $\T_{n,\ell}$ and $C(Q) = \seq{c_1;c_2;\ldots;c_n}$ be the sequence of all cross sections in~$Q$. 
To compute the exact distance of the sequence patterns in~$\rep{Q}$, we construct an $\ell\times\ell$ matrix $A_Q$ as follows. 
For any cross section $c = \seq{a,b} \in \C_{2,\ell}$ where $a,b\in L$, if $c$ appears $m$ times in $C(Q)$, then the entry of $A_Q$ at row $a$ and column $b$ is $A_Q(a,b) = m$.
Moreover, any sequence of permutations $\Phi = \seq{\phi_1,\phi_2} \in S_\ell^2$ can be written as a permutation matrix~$P_\Phi$, where $\phi_1$ permutes rows and $\phi_2$ permutes columns.
Note that the diagonal entries of $A_Q$ record the number of constant cross sections in~$Q$.
Thus Formula~(\ref{diffx}) can be written as Formula~(\ref{diffm}).

\begin{equation}
d(\rep{Q}) = n - \max_{\Phi \in S_\ell^2} \mathrm{tr}(P_\Phi A_Q)
\label{diffm}
\end{equation}

\noindent
This is equivalent to the linear assignment problem, and we can use the Hungarian algorithm or Kuhn-Munkres algorithm to compute the exact distance between two sequence patterns in polynomial time~\cite{Kuhn1955,Munkres1957}.
In general, computing the exact distance of a set $T = \rep{Q} = \seq{\rep{q_1},\rep{q_2},\ldots,\rep{q_k}}$ of sequence patterns in $\T_{n,\ell}$ is equivalent to the $k$-dimensional assignment problem, and algorithms to solve the assignment problem, for example in~\cite{Li2021}, can be used to compute the exact distance of sequence patterns.

\section{Discussion}\label{sec4}

We have formally defined the Hamming distance for $k$ sequence patterns of length $n$ and level $\ell$.
We have computed the maximal Hamming distance of $k$ sequence patterns of length $n$ and level $\ell$ and demonstrated how to compute the exact distance between two sequence patterns of length $n$ and level $\ell$.
There remains considerable scope to modify the definitions and ask various questions.
For example, we can impose restrictions such as all symbols must be present in a sequence or present symbols must appear twice;
inspired by coloring problems, we can require that adjacent elements in a sequence must have distinct symbols.
With these constraints, what is the maximal distance for such sequence patterns?
Lastly, it would also be interesting to check if the distance for $k$ sequence patterns is a metric under some generalization of the triangle inequality, which can involve hypergraphs~\cite{Friedgut2004}.

\section*{Implementation}
Code to compute the numbers and the distances of sequence patterns is available at \url{https://github.com/pliumath/sequence-patterns}

\section*{Acknowledgements}
P.L. was partially supported by the grant of the Federal Government of Canada's Canada 150 Research Chair program to Prof.~C. Colijn and by the National Science Foundation DMS/NIGMS award \#2054347 to Prof. M. V\'azquez. 
J.N. was supported by a doctoral scholarship from the China Scholarship Council. 

\bibliographystyle{plain}
\bibliography{references.bib}

\end{document}